\documentclass[12pt]{amsart}

\title[Lengths of closed geodesics on random surfaces of large genus]{Lengths of closed geodesics\\ on random surfaces of large genus}
\author{Maryam Mirzakhani and Bram Petri}
\date{\today}

\usepackage{amsmath,amsfonts,amssymb,amsthm,graphicx,overpic,hyperref}
\usepackage{float,enumitem}
\usepackage[toc,page]{appendix}

\newtheorem{thm}{Theorem}[section]
\newtheorem{prp}[thm]{Proposition}

\newtheorem{lem}[thm]{Lemma}

\newtheorem{innerthmrep}{Theorem}
\newenvironment{thmrep}[1]
  {\renewcommand\theinnerthmrep{#1}\innerthmrep}
  {\endinnerthmrep}

\theoremstyle{definition}

\usepackage[margin=1.35in]{geometry}
\newcommand{\nc}{\newcommand}
\nc{\dmo}{\DeclareMathOperator}

\nc{\abs}[1]{\left| #1 \right|}
\nc{\bigO}[1]{\mathcal{O}\left(#1\right)}
\nc{\card}[1]{\left|#1\right|}
\nc{\ceil}[1]{\left\lceil #1 \right\rceil}
\nc{\CC}{\mathbb{C}}
\nc{\floor}[1]{\left\lfloor #1 \right\rfloor}
\nc{\ZZ}{\mathbb{Z}}
\nc{\len}[1]{\left| #1 \right|}
\nc{\littleo}[1]{o\left(#1\right)}
\dmo{\Mat}{Mat}
\nc{\NN}{\mathbb{N}}
\nc{\norm}[1]{\left|\left| #1 \right|\right|}
\nc{\QQ}{\mathbb{Q}}
\nc{\RR}{\mathbb{R}}
\nc{\st}[2]{\left\{ #1 ;\; #2\right\}}
\dmo{\supp}{supp}
\nc{\tr}[1]{\mathrm{tr}\left(#1\right)}

\dmo{\area}{area}
\dmo{\conv}{conv}
\dmo{\diam}{diam}
\dmo{\dist}{\mathrm{d}}
\nc{\HH}{\mathbb{H}}
\dmo{\MCG}{MCG}
\dmo{\MPL}{MPL}
\dmo{\Mod}{\mathcal{M}}
\dmo{\PL}{PL}
\nc{\Sphere}{\mathbb{S}}
\dmo{\sys}{sys}
\dmo{\Teich}{\mathcal{T}}
\nc{\Torus}{\mathbb{T}}
\dmo{\vol}{vol}
\dmo{\WP}{WP}

\dmo{\convTV}{\;\stackrel{\mathrm{TV}}{\longrightarrow}\;}
\nc{\ExV}[2]{\mathbb{E}_{#1}\left[#2\right]}
\dmo{\EE}{\mathbb{E}}
\nc{\Pro}[2]{\mathbb{P}_{#1}\left[#2\right]}
\dmo{\PP}{\mathbb{P}}
\nc{\distTV}[2]{\mathrm{d}_{\rm TV}\left(#1,#2\right)}
\dmo{\UU}{\mathbb{U}}
\nc{\Var}[2]{\mathbb{V}\mathrm{ar}_{#1}\left[#2\right]}

\dmo{\alt}{\mathfrak{A}}
\dmo{\Aut}{Aut}
\dmo{\Fix}{Fix}
\dmo{\Hom}{Hom}
\dmo{\PSL}{PSL}
\dmo{\Rep}{Rep}
\dmo{\sym}{\mathfrak{S}}

\begin{document}

\begin{abstract} We prove Poisson approximation results for the bottom part of the length spectrum of a random closed hyperbolic surface of large genus. Here, a random hyperbolic surface is a surface picked at random using the Weil-Petersson volume form on the corresponding moduli space. As an application of our result, we compute the large genus limit of the expected systole. 
\end{abstract}

\maketitle

%%%%%%%%%%%%%%%%%%%%%%%%%%%%%%%%%%%%%%%%%%%%%%%%%
%		I N T R O D U C T I O N
%%%%%%%%%%%%%%%%%%%%%%%%%%%%%%%%%%%%%%%%%%%%%%%%%
\section{Introduction}

In this paper, we study the distribution of short closed geodesics on random hyperbolic surfaces. Our definition of a random surface is as follows. First of all, we consider for every $g\geq 2$ the moduli space $\Mod_g$ of closed hyperbolic surfaces of genus $g$. The corresponding Teichm\"uller space $\Teich_g$ comes with a symplectic form $\omega_g$, called the Weil-Petersson symplectic from. The associated volume form descends to $\Mod_g$ and is of finite total volume. This means that we obtain a probability measure $\PP_g$ on $\Mod_g$ by defining
\[\PP_g[A] = \frac{\vol_{\WP}(A)}{\vol_{\WP}(\Mod_g)}\]
for every measurable set $A\subseteq \Mod_g$, where $\vol_{\WP}(A)$ denotes the Weil-Petersson volume of $A$. Our main goal is now to combine methods from probability theory and Weil-Petersson geometry to estimate probabilities of the form 
\[\PP_g[X\in \Mod_g \text{ has }k\text{ closed geodesics of length }\leq L].\]
In particular, we will compute the large genus limits of these probabilities. For example, we determine which proportion of the Weil-Petersson volume is asymptotically taken up by the $\varepsilon$-thin part of moduli space. 

\subsection*{New results} Before we state our results, we need to set up some notation. Given $X\in\Mod_g$ and an interval $[a,b]\subset \RR_+$, let $N_{g,[a,b]}(X)$ denote the number of primitive closed geodesics on $X$ with lengths in the given interval. Note that in our setup 
\[N_{g,[a,b]}:\Mod_g\to\NN\]
may be interpreted as a random variable. 

In Section \ref{sec_lengthspectrum} we prove:
\begin{thmrep}{\ref{thm_lengthspectrum}} Let $[a_1,b_1], [a_2,b_2], \ldots [a_k, b_k] \subset \RR_+$ be disjoint intervals. Then, as $g\to \infty$, the vector of random variables
\[ (N_{g,[a_1,b_1]},\ldots, N_{g,[a_k,b_k]}):\Mod_g \to \NN^k \]
converges jointly in distribution to a vector of independent Poisson distributed random variables with means $\lambda_{[a_i,b_i]}$, where
\[ \lambda_{[a_i,b_i]} = \int_{a_i}^{b_i} \frac{e^t+e^{-t}-2}{2t}dt. \]
for $i=1,\ldots,k$.
\end{thmrep}
Recall that a random variable $Z:\Omega\to\NN$ on a probability space $\Omega$ is said to be Poisson distributed with mean $\lambda\in (0,\infty)$ if $\PP[Z=k] = \lambda^k e^{-\lambda} / k!$ for all $k\in\NN$. This means that Theorem \ref{thm_lengthspectrum} allows for the explicit computation of large genus limits of the probabilities we are after. For example, we have
\[\lim_{g\to\infty} \PP_g[\sys(X) \leq \varepsilon]  = 1 - e^{-\lambda_{[0,\varepsilon]}} \sim 1-e^{-\varepsilon^2 / 4}\sim \frac{\varepsilon^2}{4} \]
as $\varepsilon\to 0$.

As an application of Theorem \ref{thm_lengthspectrum}, we study the distribution of the \emph{systole} - the length of the shortest closed geodesic - of a random surface. Recall that the expected value of a random variable $Z:\Mod_g \to \RR$ is given by 
\[\EE_g[Z] = \frac{\int_{\Mod_g}Z(X) dX}{\int_{\Mod_g} dX},\]
where the integrals are with respect to the Weil-Petersson volume form. We show (Section \ref{sec_systole}):
\begin{thmrep}{\ref{thm_systole}} Given $X\in\Mod_g$, let $\sys(X)$ denote its systole. We have:
\[\lim\limits_{g\to\infty} \EE_g[\sys] = \int_0^\infty e^{-\lambda_{[0,R]}}dR =1.61498\ldots.\]
\end{thmrep}

\subsection*{Notes and references} The number of closed geodesics up to a given length on a hyperbolic surface has been investigated by many authors. A classical result due to Delsarte, Huber and Selberg states that for any closed hyperbolic surface $X$ we have
\[N_{g,[0,L]}(X) \sim e^{L}/L\]
as $L\to \infty$, see for instance \cite{Bus} for details. More recently, the number of \emph{simple} closed geodesics of length $\leq L$ was shown to be asymptotic to $c_X\cdot L^{6g-6}$ \cite{Mir5}, where $c_X$ is a constant depending only on $X$. Similar results are known to hold for the number of curves up to a given length in a fixed mapping class group orbit \cite{Mir6,ES}. 

The study of the Weil-Petersson geometry of moduli spaces of surfaces of large genus also has a long history. Estimates on Weil-Petersson volumes of moduli spaces of surfaces of large genus were derived by Penner \cite{Pen}, Grushevsky \cite{Gru}, Schumacher and Trappani \cite{ST} and Mirzkhani and Zograf \cite{MZ}. The large genus behavior of the Weil-Petersson diameter has been exhibited by Cavendish and Parlier \cite{CP},
the growth of the number of balls needed to cover the thick part of moduli space by Fletcher, Kahn and Markovic \cite{FKM} and the injectivity radius by Wu \cite{Wu}.

Closed geodesics on random surfaces have been studied in \cite{Mir3,Mir4,GPY}. In \cite{Mir3}, it was for example proven that for every fixed genus $g$, $\PP[X\in\Mod_g \text{ has }\sys(X)<\varepsilon]$ decays like $\varepsilon^2$ as $\varepsilon\to 0$, independently of the genus. In \cite{Mir4}, the large genus limits of the expected values $\EE_g[N_{g,[0,L]}]$ were computed and it was also already shown that the expected systole is bounded independently of the genus (see Section \ref{sec_backgroundrandom} and Appendix \ref{app_workmaryam} for more details). In \cite{GPY}, Guth, Parlier and Young proved that random surfaces do not admit pants decompostions of total length less than roughly $g^{7/6}$. This contrasts Theorem \ref{thm_systole} above, which guarantees that the probability that the surface has short curves does not tend to zero. These short curves will however generically not form a pants decomposition.

\subsection*{Idea of the proof} The proof of Theorem \ref{thm_lengthspectrum} relies on a combination of the method of moments for Poisson approximation and results on Weil-Petersson volumes of large genus surfaces. Our final goal in proving Theorem \ref{thm_lengthspectrum} is to control the factorial moments of the functions $N_{g,[a,b]}:\Mod_g\to \NN$. These are expressions of the form
\[\EE[N_{g,[a,b]}(N_{g,[a,b]}-1) (N_{g,[a,b]}-2)\cdots (N_{g,[a,b]}-k+1)],\]
for $k\in\NN$. Indeed, it is the content of the method of moments that once these moments are shown to converge to the moments of a Poisson distributed random variable, Theorem \ref{thm_lengthspectrum} follows. 

To control these moments, our first step is to use the integration techniques developed in \cite{Mir1}. These allow us to express the joint factorial moments in terms of Weil-Petersson volumes of moduli spaces of hyperbolic surfaces with boundary components. These are defined as follows. Given $L=(L_1,\ldots,L_n) \in \RR_+^n$, we let $\Mod_{g,n}(L)$ be the moduli space of hyperbolic surfaces of genus $g$ with $n$ boundary components of lengths $L_1,\ldots L_n$ respectively. The corresponding Teichm\"uller space $\Teich_{g,n}(L)$ also carries a symplectic form $\omega_{g,n}$, called the Weil-Petersson symplectic form \cite{Gol,Wol1}, that descends to $\Mod_{g,n}(L)$. It turns out that the Weil-Petersson volume $V_{g,n}(L)$ of $\Mod_{g,n}(L)$ is a polynomial in $L_1^2,\ldots,L_n^2$ of degree $3g+n-3$ \cite{Mir1}. Moreover, the coefficients of these polynomials can be computed in terms of integrals of the form
\[ \int_{\overline{\Mod}_{g,n}} \psi_1^{d_1}\cdots \psi_n^{d_n}\omega_{g,n}^{3g+n-3-\sum_{i=1}^n d_i}, \]
where $\psi_i\in H^2(\Mod_{g,n},\QQ)$ is the first Chern class of the $i^{th}$ tautological line bundle on $\Mod_{g,n}$ for all $1\leq i \leq n$ \cite{Mir2} (see Section \ref{sec_backgroundIntersection}). Using bounds on Weil-Petersson volumes and integrals of these Chern classes from \cite{Mir3, MZ}, we can single out the main contributions to the joint factorial moments and prove the convergence we are after. Finally, Theorem \ref{thm_systole} follows from Theorem \ref{thm_lengthspectrum} together with a dominated convergence argument.

\subsection*{Remark}
In \cite{BM}, Brooks and Makover defined a more combinatorial model for random hyperbolic surfaces based on Bely\v{\i} surfaces. The idea of this model is to glue an even number of ideal hyperbolic triangles together along their sides without shear and then conformally compactify the resulting surface. As such, this model gives rise to closed hyperbolic surfaces and it follows from a theorem due to Bely\v{\i} \cite{Bel} that the resulting set of surfaces is dense in $\Mod_g$ for every $g\geq 2$.

If we denote the number of triangles by $2N$, then the expected genus of these surfaces is asymptotic to $N/2$ (see \cite{BM,Gam}). It thus seems natural to compare the geometry of these random surfaces to that of the random surfaces studied in this article of the corresponding genus. It turns out that many of the known results are very similar, even though no a priori reason for this is known.

The number of short curves is also known to be asymptotically Poisson distributed in the model introduced by Brooks and Makover \cite{Pet,PT}. Moreover, the expected systole also converges to a constant (as $N\to\infty$) \cite{Pet}. On the other hand, both the means of the corresponding Poisson distribution and the limit of the expected systole are slightly different. 

Finally, it follows from Brooks and Makover's results \cite{BM} that their random surfaces have the property that
\[ \PP_N[\sys \geq b] \to 1\]
as $N\to\infty$, where $b=2\cdot\cosh^{-1}(3/2)$. In other words, the probability measure from Brooks and Makover's model asymptotically concentrates in the $b$-thick part of moduli space. Instead, Theorem \ref{thm_lengthspectrum} immediately implies that in our setting
\[\lim_{g\to\infty}\PP_g[\sys \geq b] = \lim_{g\to\infty}\PP_g[N_{g,[0,b]}=0] = e^{-\lambda_{[0,b]}} = 0.339043\ldots.\]

\subsection*{About this paper}
The mathematical content of this paper is a joint project. However, the writing was done by the second author and completed after Maryam Mirzakhani passed. As such, she wasn't able to consent to its publication.

Part of this work uses unpublished results by Maryam Mirzakhani. Because these results are currently not available elsewhere, Appendix \ref{app_workmaryam} contains a brief sketch on how to prove the results we need. We stress that the results in this appendix are due to Mirzakhani.

\subsection*{Acknowledgement} The second named author acknowledges support from the ERC Advanced Grant ``Moduli'' and the Max Planck Institute for Mathematics in Bonn. He would like to thank Alex Wright for useful discussions. He would also like to thank Stanford University for its hospitality when he visited in spring 2016. It was a great privilege spending a week discussing mathematics with Maryam.

%%%%%%%%%%%%%%%%%%%%%%%%%%%%%%%%%%%%%%%%%%%%%%%%%
%		B A C K G R O U N D
%%%%%%%%%%%%%%%%%%%%%%%%%%%%%%%%%%%%%%%%%%%%%%%%%
\section{Background}

In this section we present some background material and set up notation. For more details, we refer the reader to \cite{Bus,Mir1,Wol2}.

\subsection{Teichm\"uller and moduli spaces}
In what follows, $\Sigma_{g,n}$ will denote a surface of genus $g$ and $n$ boundary components. We will write $\Sigma_g=\Sigma_{g,0}$.

Given $g,n \in \NN$ and $L=(L_1,\ldots,L_n)\in\RR_+$, we define the Teichmu\"uller space $\Teich_{g,n}(L)$ (or $\Teich(\Sigma_{g,n},L)$) to be the space 
\[\Teich_{g,n}(L)=\st{(X,f)}{\begin{array}{c} X\text{ a complete hyperbolic surface with totally} \\ \text{geodesic boundary components, whose lengths}  \\
 \text{are given by }L, f:\Sigma_{g,n}\to X\text{ a diffeomorphism}\end{array}}/\sim. \]
Here, the equivalence relation is given by:
\[(X,f)\sim (X',f') \text{ if and only if } f'\circ f^{-1}: X\to X' \]
is isotopic to an isometry. When $L_i=0$ then the corresponding boundary component is assumed to be a cusp. We set $\Teich_{g,n}=\Teich_{g,n}(0,\ldots,0)$, $\Teich_g = \Teich_{g,0}$ and
\[\Teich\left(\bigsqcup_{i=1}^k \Sigma_{g_i,n_i},L(1),\ldots, L(k)\right) = \prod_{i=1}^k \Teich_{g_i,n_i}(L(i)),\]
for all $L(1)\in\RR_+^{n_1},\ldots,L(k)\in\RR_+^{n_k}$.

The mapping class group $\MCG(\Sigma_{g,n})$ of isotopy classes of orientation preserving diffeomorphisms that setwise fix the boundary acts on $\Teich_{g,n}(L)$ and the quotient
\[\Mod_{g,n}(L) = \Teich_{g,n}(L) / \MCG(\Sigma_{g,n})\]
is the moduli space of Riemann surfaces homeomorphic to $\Sigma_{g,n}$. We will again write $\mathcal{M}_{g,n}=\mathcal{M}_{g,n}(0,\ldots,0)$, $\mathcal{M}_g = \mathcal{M}_{g,0}$ and
\[\Mod\left(\bigsqcup_{i=1}^k \Sigma_{g_i,n_i},L(1),\ldots,L(k)\right) = \prod_{i=1}^k \Mod_{g_i,n_i}(L(i)),\]
for all $L(1)\in\RR_+^{n_1},\ldots,L(k)\in\RR_+^{n_k}$.

Goldman \cite{Gol} showed that the space $\Teich_{g,n}(L)$ carries a $\MCG(\Sigma_{g,n})$-invariant symplectic form $\omega_{g,n}$, called the Weil-Petersson symplectic form. When the lengths $L_i=0$ for all $1\leq i\leq n$, then  $\omega_{g,n}$ is a symplectic form coming from a K\"ahler metric on $\Mod_{g,n}$ (see \cite{IT} for details).

\subsection{Length functions and Fenchel-Nielsen coordinates} Classical hyperbolic geometry (see for instance \cite[Theorem 1.6.6.]{Bus}) implies that in the free homotopy class of a (non null- or boundary-homotopic) simple closed curve on a hyperbolic surface there exists a unique simple closed geodesic which minimizes length over the homotopy class. As such, every free homotopy class $\alpha$ of non-trivial simple closed curves on $\Sigma_{g,n}$ defines a function $\ell_\alpha:\Teich_{g,n}(L)\to \RR_+$ that assigns to $X$ the length of the unique geodesic on $X$ freely homotopic to $\alpha$.

Let $\mathcal{P}=\{\alpha_i\}_{i=1}^{3g+n-3}$ be a pants decomposition - a set of simple closed curves so that $\Sigma_{g,n}\smallsetminus\cup_i \alpha_i$ is a disjoint union of three-holed spheres (pairs of pants) - of $\Sigma_{g,n}$. Given a point $X\in\Teich_{g,n}(L)$, $\mathcal{P}$ is homotopic to a unique set of simple closed geodesics on $X$. This allows us to assign two real numbers to each curve $\alpha_i$ in $\mathcal{P}$: the length $\ell_{\alpha_i}(X)$ of the corresponding geodesic and the twist $\tau_{\alpha_i}(X)$ of the gluing at that geodesic. For any pants decomposition $\mathcal{P}$ the map 
\[\Teich_{g,n}(L) \to \left(\RR_+\times \RR\right)^\mathcal{P}\]
given by
\[X\mapsto (\ell_{\alpha_i}(X),\tau_{\alpha_i}(X))_{i=1}^{3g+n-3}\]
is a global coordinate system for $\Teich_{g,n}(L)$, called \emph{Fenchel-Nielsen coordinates} (see for instance \cite[Chapter 6]{Bus}). In \cite{Wol3}, Wolpert proved that $\omega_{g,n}$ has a particularly nice form in these coordinates:
\begin{thm}\label{thm_wolpert} \cite{Wol3} The Weil-Petersson symplectic form on $\Teich_{g,n}(L)$ is given by
\[ \omega_{g,n} = \sum_{i=1}^{3g+n-3} d\ell_{\alpha_i}\wedge d\tau_{\alpha_i}.\] 
\end{thm}

\subsection{Integrating geometric functions on $\mathcal{M}_{g,n}$}
First of all, let us write $V_{g,n}(L)$ for the Weil-Petersson volume of $\Mod_{g,n}(L)$. That is
\[V_{g,n}(L) = \int_{\Mod_{g,n}(L)} \frac{\wedge^{3g+n-3}\omega_{g,n}}{(3g+n-3)!} .\]
We will write $V_{g,n} = V_{g,n}(0,\ldots,0)$ and $V_g=V_{g,0}$.

Given a function $F:\mathbb{R}_+^k\to\mathbb{R}$ and a $k$-tuple of curves $\Gamma=(\gamma_1,\ldots,\gamma_k)$ on $\Sigma_{g,n}$ we define
\[F^\Gamma:\Mod_g\to\mathbb{R}  \]
by
\[F^\Gamma(X)= \sum_{(\alpha_1,\ldots,\alpha_k)\in \MCG(\Sigma_g)\cdot\Gamma}F(\ell_X(\alpha_1),\ldots,\ell_X(\alpha_k)).\]

Let $N(\Gamma)\subset \Sigma_g$ denote a regular neighborhood of $\Gamma$. If $\Sigma_{g}\setminus N(\Gamma) = \bigsqcup\limits_{i=1}^r\Sigma_{g_i,n_i}$ and $x\in\mathbb{R}_+^k$ then we write
\[V_{g}(\Gamma,x) = \prod_{i=1}^r V_{g_i,n_i}(x_{i,1},\ldots,x_{i,n_i})\]
where the $x_{i,s}$ are so that when the boundary of $\Sigma_{g_i,n_i}$ consists of the curves $\{\gamma_{t_1},\ldots \gamma_{t_{n_i}}\}$ then $\{x_{i,s}\}_s=\{x_{t_s}\}_s$. Note that every $x_i$ appears twice.

The integral of $F^\Gamma$ over $\Mod_g$ can now be computed as follows:
\begin{thm} \label{thm_integration} \cite{Mir1} Given $\Gamma=(\gamma_1,\ldots,\gamma_k)$ a $k$-tuple of simple closed curves on $\Sigma_{g}$ and $F:\mathbb{R}_+^k\to\mathbb{R}$, we have
\[\int_{\Mod_g} F^\Gamma(X) dX = C_\Gamma\cdot \int_{\mathbb{R}_+^k}F(x)V_{g,n}(\Gamma,x) x_1\cdots x_k dx_1\wedge \cdots \wedge dx_k,\]
where $C_\Gamma$ is a constant depending on $\Gamma$ only. If $\Sigma_g\setminus\Gamma$ is connected, then $C_\Gamma=2^{-k}$.
\end{thm}

\subsection{Connection with intersection numbers}\label{sec_backgroundIntersection}

It turns out that Weil-Petersson volumes of moduli spaces of surfaces with boundary components can be related to intersection numbers on $\overline{\Mod}_{g,n}$. These numbers are defined using the so-called tautological line bundles $\mathcal{L}_1,\ldots,\mathcal{L}_n$ on  $\overline{\Mod}_{g,n}$. The fiber of the bundle $\mathcal{L}_i$ at $X\in\overline{\Mod}_{g,n}$ is the cotangent space at the $i^{th}$ marked point on $X$. 

Now let $\psi_i=c_1(\mathcal{L}_i) \in H^2(\Mod_{g,n},\QQ)$ denote the first Chern class of the $i^{th}$ tautological line bundle. Given $d=(d_1,\ldots,d_n)\in\NN^n$, we will write \linebreak $\abs{d}=d_1+\ldots+d_n$. If $\abs{d}\leq 3g+n-3$, we write
\[ [\tau_{d_1},\ldots,\tau_{d_n}]_{g,n} := \frac{\prod\limits_{i=1}^n (2d_i+1)! 2^{\abs{d}}}{\prod\limits_{i=0}^n d_i!} \int_{\overline{\mathcal{M}}_{g,n}}\psi_1^{d_1} \cdots \psi_n^{d_n}\omega^{d_0} \]
where $d_0=3g-3+n-\abs{d}$. See \cite{AC} for more details.

The volumes of moduli spaces of surfaces with boundary can now be expressed as follows:
\begin{thm}\label{thm_volformula} \cite{Mir2} Let $g,n\in\mathbb{N}$ and $x_1,\ldots,x_n\in \mathbb{R}_+$. Then
\[V_{g,n}(2x_1,\ldots,2x_n) = \sum_{d\in \mathbb{N}^n,\;\abs{d}\leq 3g+n-3}[\tau_{d_1},\ldots,\tau_{d_n}]_{g,n} \frac{x_1^{2d_1}}{(2d_1+1)!}\cdots \frac{x_n^{2d_n}}{(2d_n+1)!}.  \]
\end{thm}

\subsection{Bounds on volumes and intersection numbers}

In order to estimate moments in Section \ref{sec_lengthspectrum}, we will need bounds on both Weil-Petersson volumes and intersection numbers. Various estimates on both are known and we refer the reader to \cite{MZ} and the references therein for these. We will state only those bounds that we need.

The bound on intersection numbers we will need, which can be found on \cite[Page 286]{Mir3}, is the following:
\begin{lem}\label{lem_bound} \cite{Mir3} Given $n\in\mathbb{N}$, there exists a constant $c_0>0$ independent of $g$ and $d$ such that:
\[0\leq 1- \frac{[\tau_{d_1},\ldots,\tau_{d_n}]_{g,n}}{V_{g,n}} \leq c_0\frac{\abs{d}^2}{g}.\]
for all $g\in\mathbb{N}$ and $d\in\mathbb{N}^n$.
\end{lem}

The following is part of Theorem 1.4 in \cite{MZ}:
\begin{thm}\label{thm_volestimate2}\cite{MZ} for any fixed $n\geq 0$: 
\[\frac{V_{g-1,n+2}}{V_{g,n}} = 1 +\frac{3-2n}{\pi^2}\cdot \frac{1}{g} + O\left(\frac{1}{g^2}\right) \]
as $g\to\infty$.
\end{thm}
 
The asymptotic behavior of $V_{g,n}$ as $g\to\infty$ is known up to a multiplicative constant:
\begin{thm}\label{thm_volexpansion}\cite{MZ} There exists a universal constant $\alpha\in (0,\infty)$ such that for any given $k\geq 1$, $n\geq 0$,
\[V_{g,n} = \alpha\cdot \frac{(2g-3+n)! (4\pi^2)^{2g-3+n}}{\sqrt{g}}\left(1+\frac{c_n^{(1)}}{g}+\ldots +\frac{c_n^{(k)}}{g^k} + O\left(\frac{1}{g^{k+1}} \right) \right) \]
as $g\to\infty$. Each term $c_n^{(i)}$ in the asymptotic expansion is a polynomial in $n$ of degree $2i$ with coefficients in $\mathbb{Q}[\pi^{-2},\pi^2]$ that are effectively computable.
\end{thm}

Finally, we have \cite[Equation 3.7]{Mir3}:
\begin{lem}\label{lem_volestimate3} Let $g,n\in\NN$ and $x_1,\ldots,x_n \in \RR^+$. Then
\[ V_{g,n}(2x_1,\ldots,2x_n)\leq e^x\cdot V_{g,n}, \]
where $x=\sum\limits_{i=1}^nx_i$.
\end{lem}

\subsection{Random surfaces}\label{sec_backgroundrandom}

Because $V_g$ is finite for every $g$, we can turn $\Mod_g$ into a probability space. Indeed, given a measurable subset $A\subset\Mod_g$, define
\[\PP_g[A] = \frac{1}{V_g} \int_A \frac{\wedge^{3g-3}\omega}{(3g-3)!}  \]
and given a random variable (i.e. a measurable function) $F:\Mod_g\to\RR$ and $B\subset\mathbb{R}$, we define
\[\EE_g[F] = \frac{1}{V_g}\int_{\Mod_g}F(X)\; dX \;\;\text{and}\;\; \PP_g[F\in B] =  \EE_g[\chi_{\{F\in B\}}]\] 
where the integral on the left hand side is shorthand for integration with respect to $\frac{\wedge^{3g-3+n}\omega}{(3g-3+n)!}$, and might not be finite, and $\chi_{\{F\in B\}}:\mathcal{M}_g\to\mathbb{R}$ is defined by
\[ \chi_{\{F\in B\}}(X) = \left\{\begin{array}{ll} 1 & \text{if } F(X)\in B \\ 0 & \text{otherwise} \end{array} \right. \]

In order to compute the expected systole later on, we need the following bound, due to Mirzakhani \cite{Mir4}:
\begin{thm}\label{thm_sysbd}\cite{Mir4}  There exist universal constants $A, B>0$ so that for any sequence $\{c_g\}_g$ of postive numbers with $c_g<A\; \log(g)$ we have:
\[\PP_g[\text{The systole of }S\text{ has length }>c_g] <   B\;c_g \; e^{-c_g}. \]
\end{thm}
Because it is not available elsewhere in the literature, we sketch a proof of this result in Appendix \ref{app_workmaryam}.

\subsection{The method of moments} 
The main probabilistic tool we use is the method of moments. More precisely, we will study the asymptotics of the joint factorial moments of finite collections of sequences of random variables.

Let us first set up some notation. Given a probability space $(\Omega,\Sigma,\mathbb{P})$, a random variable $X:\Omega\to\mathbb{N}$ and $n\in\mathbb{N}$, we define the random variable
\[(X)_n := X(X-1)\cdots (X-n+1) \]
If its expectation $\EE[(X)_n]$ exists, it is called the \emph{$n^{th}$ factorial moment} of $X$.

Furthermore, recall that an $\mathbb{N}$-valued random variable $X$ is said to be Poisson distributed with mean $\lambda\in [0,\infty)$ if 
\[\PP[X=k] = \frac{\lambda^k e^{-\lambda}}{k!}\;\;\text{for all}\;\;k\in\mathbb{N} \]
The $n^{th}$ factorial moment of such a variable is equal to $\lambda^n$ for all $n\in\mathbb{N}$. It turns out that this determines the law of the variable: an $\mathbb{N}$-valued random variable is Poisson distributed with mean $\lambda\in (0,\infty)$ if and only if its $n^{th}$ factorial moment is equal to $\lambda^n$ for all $n\in\NN$.

We are now able to state the theorem we need, which can for instance be found in \cite[Theorem 21]{Bol}.

\begin{thm}\label{thm_poisson} \emph{(The method of moments)} Let $\{(\Omega_i,\Sigma_i,\mathbb{P}_i)\}_{i\in\mathbb{N}}$ be a sequence of probability spaces. Furthermore, let $k\in\mathbb{N}$, let $X_{1,i},\ldots,X_{k,i}:\Omega_i\to\mathbb{N}$ be random variables and suppose there exist $\lambda_1,\ldots,\lambda_k \in (0,\infty)$ such that
\[\lim_{i\to\infty}\EE[(X_{1,i})_{n_1}(X_{2,i})_{n_2}\cdots (X_{k,i})_{n_k}] = \lambda_1^{n_1}\lambda_2^{n_2}\cdots\lambda_k^{n_k}\] 
for all $n_1,\ldots,n_k\in\mathbb{N}$ then 
\[\lim_{i\to\infty}\PP[X_{1,i}=m_1,\ldots,X_{k,i}=m_k] = \prod_{j=1}^k \frac{\lambda_j^{m_j}e^{-\lambda_j}}{m_j!}, \]
for all $m_1,\ldots,m_k\in\mathbb{N}$. In other words, the vector $(X_{1,i},\ldots,X_{k,i}):\Omega\to\NN^k$ converges jointly in distribution to a vector of independent Poisson variables with means given by $\lambda_1,\ldots,\lambda_k$.
\end{thm}

Given random variables $X_{1,i},\ldots,X_{k,i}:\Omega_i\to\mathbb{N}$ for all $i\in\mathbb{N}$ and random variables $X_1,\ldots,X_k:\Omega'\to\mathbb{N}$, we will use the shorthand notation 
\[X_{1,i},\ldots,X_{k,i}\stackrel{d}{\longrightarrow} X_1,\ldots,X_k\]
(for joint convergence in distribution) to indicate that:
\[\lim_{i\to\infty}\PP[X_{1,i}=m_1,\ldots,X_{k,i}=m_k] = \PP[X_1=m_1,\ldots,X_k=m_k]  \]
for all $m_1,\ldots,m_k\in\mathbb{N}$.

\section{Estimates on Weil-Petersson volumes}

Before we get to the proof of the main theorems, we will need to derive some estimates on Weil-Petersson volumes. All of these are mild generalizations of estimates that were already known.

We will need a comparison between the volume of a moduli space of surfaces with boundary components and that of a moduli space of surfaces with cusps. In the case of two boundary components, this can be found in \cite{Mir4}. For completeness, we include a proof sketch. The strategy is the same as in \cite{Mir4}.
\begin{prp}\label{prp_volestimate1} Let $g,n\in\mathbb{N}$ and $x_1,\ldots,x_n\in\mathbb{R}_+$ then
\[\frac{V_{g,n}(2x_1,\ldots,2x_n)}{V_{g,n}} = \prod\limits_{i=1}^n \frac{\sinh(x_i)}{x_i}\left(1 + O\left(\frac{\prod\limits_{i=1}^n x_i}{g}\right) \right) \]
as $g\to\infty$.
\end{prp}

\begin{proof}[Proof sketch] From Theorem \ref{thm_volformula} and Lemma \ref{lem_bound} we obtain
\[\sum_{\substack{d\in \mathbb{N}^n, \\ \abs{d}\leq 3g+n-3}}\prod_{i=1}^n\frac{x_i^{2d_i}}{(2d_i+1)!}-\frac{V_{g,2n}(2x)}{V_{g,2n}}\leq  \frac{c_0}{g}\sum_{\substack{d\in \mathbb{N}^n, \\ \abs{d}\leq 3g+n-3}}  \abs{d}^2\prod_{i=1}^n\frac{x_i^{2d_i}}{(2d_i+1)!}\]
In the first term on the left hand side we recognize the beginning of the Taylor expansion of $\prod\limits_{i=1}^n \frac{\sinh(x_i)}{x_i}$. The expression on the right hand side is of the order $O\left(\frac{1}{g}\prod\limits_{i=1}^n \sinh(x_i)\right)$.
\end{proof}

Finally, we need a version of \cite[Lemma 3.3]{Mir3} with more variables:
\begin{lem}\label{lem_volestimate4} Let $q,K\in\mathbb{N}$ and $n_1,\ldots,n_q\in\mathbb{N}\setminus\{0\}$ such that $\sum\limits_{i=1}^q n_i=2K$, then:
\[\sum_{\{g_i\}}V_{g_1,n_1}\times\ldots\times V_{g_q,n_q} = O\left(\frac{V_g}{g^{q-1}}\right)  \]
as $g\to\infty$, where the sum is over all multisets $\{g_i\}_{i=1}^q\subset\mathbb{N}$ such that $\sum\limits_{i=1}^q g_i = g +q-K-1$ and $2g_i-3+n_i\geq 0$ for all $i=1,\ldots q$.
\end{lem}

\begin{proof} This lemma is a direct application of Theorem \ref{thm_volexpansion}. In fact, Lemma 3.3 from \cite{Mir3} is used in the proof of this theorem. 

Theorem \ref{thm_volexpansion} tells us that there exists some $a>0$ independent of all $g_i$ and $g$, $q$ and $K$ such that
\[\frac{1}{V_g}\sum_{\{g_i\}}V_{g_1,n_1}\times\ldots\times V_{g_q,n_q} \leq a^q \sum_{\{g_i\}}\frac{\sqrt{g}\prod\limits_{i=1}^q (2g_i-3+n_i)! (4\pi^2)^{2g_i-3+n_i}}{(2g-3)! (4\pi^2)^{2g-3} \prod\limits_{i=1}^q \max\{\sqrt{g_i},1\}} \]
Stirling's approximation states that $n!$ can be uniformly bounded from above and below by constant multiples of $\sqrt{n}\left(\frac{n}{e}\right)^n$ for all $n\in \mathbb{N}\setminus \{0\}$ (see \cite{Rob}). Using this and working out the sums in the exponents of $4\pi^2$ and $e$ (the latter coming out of Stirling's approximation) we obtain that there exists a constant $b>0$ such that:
\[\frac{1}{V_g}\sum_{\{g_i\}}V_{g_1,n_1}\times\ldots\times V_{g_q,n_q} \leq
b^{q+K} \sum_{\{g_i\}}\frac{\prod\limits_{i=1}^q (2g_i-3+n_i)^{2g_i-3+n_i}}{(2g-3)^{2g-3}} = O\left(\frac{1}{g^{q-1}}\right),
\]
where the exponent $K$ comes from comparing the factors $\sqrt{2g_i-3+n_i}$ in the numerator of each term to the factors $\sqrt{g_i}$ in the denominators.
\end{proof}

\section{The length spectrum}\label{sec_lengthspectrum}

The main goal of this section is to apply Theorem \ref{thm_poisson} to the bottom part of the length spectrum of a random surface, chosen with respect to the Weil-Petersson metric.

Concretely, given $0\leq a< b\in\mathbb{R}$, we define random variables $N_{g,[a,b]}: \mathcal{M}_g\to\mathbb{N}$ by
\[N_{g,[a,b]}(X) = \card{\st{\gamma\in\mathcal{P}(X)}{\ell_X(\gamma)\in [a,b]}} \]
for all $X\in\mathcal{M}_g$, where $\mathcal{P}(X)$ denotes the set of primitive closed geodesics on $X$ and $\ell_X(\gamma)$ denotes the length of such a geodesic with respect to the metric on $X$.

Define the function $\lambda_{[\cdot,\cdot]} :[0,\infty)\times [0,\infty)\to [0,\infty)$ by
\[ \lambda_{[a,b]} = \int_a^b \frac{e^t+e^{-t}-2}{2t}dt \]
for all $a,b\in [0,\infty)$. Furthermore, given $a_1 <b_1\leq a_2 <b_2 \leq \ldots \leq a_k< b_k \in \RR_+$ let 
\[(N_{[a_1,b_1]},\ldots,N_{[a_k,b_k]}):\Omega\to\NN^k\]
denote a vector of independent Poisson distributed random variables with means $\lambda_{[a_i,b_i]}$. Here $\Omega$ is any probability space that is rich enough to carry such a variable.

We will prove the following:
\begin{thm}\label{thm_lengthspectrum} Let $a_1 <b_1\leq a_2 <b_2 \leq \ldots \leq a_k< b_k \in \RR_+$. Then
\[ (N_{g,[a_1,b_1]},\ldots, N_{g,[a_k,b_k]}) \stackrel{d}{\longrightarrow} (N_{[a_1,b_1]},\ldots N_{[a_k,b_k]})  \]
as $g\to \infty$.
\end{thm}

To prove this theorem, we will apply the method of moments (Theorem \ref{thm_poisson}). In other words, we need to estimate the joint factorial moments of the variables $N_{g,[a_i,b_i]}$.

We will prove Theorem \ref{thm_lengthspectrum} in two steps. The crucial observation (which underlies many applications of the method of moments) is that given $r_1,\ldots,r_k \in \mathbb{N}$, the random variable
\[\left(N_{g,[a_1,b_1]}\right)_{r_1}\cdots\left(N_{g,[a_k,b_k]}\right)_{r_k}: \mathcal{M}_g\to\mathbb{N}  \]
counts the number of (ordered) lists of length $k$ where the $i^{th}$ item is the number of ordered $r_i$-tuples of curves with length in $[a_i,b_i]$ on a surface in $\mathcal{M}_g$. We will write:
\[\left(N_{g,[a_1,b_1]}\right)_{r_1}\cdots\left(N_{g,[a_k,b_k]}\right)_{r_k} = Y_{g,r_1,\ldots,r_k} + Y'_{g,r_1,\ldots,r_k} \]
where $Y_{g,r_1,\ldots,r_k}$ counts the lists of tuples of \emph{simple} curves that are also all pairwise disjoint and $Y'_{g,r_1,\ldots,r_k}$ counts the lists out of which some of the curves intersect either themselves or each other. The proof now proceeds by showing that the expected value of $Y$ is asymptotic to the joint factorial moment of independent Poisson variables and that the expected value of $Y'$ tends to $0$.

We begin with $Y$:
\begin{prp}\label{prp_poissonbound1} We have
\[\EE_g[Y_{g,r_1,\ldots,r_k}] = \prod_{i=1}^k \lambda_{[a_i,b_i]}^{r_i} + O\left(\frac{\prod\limits_{i=1}^k b_i^{2r_i} e^{r_i b_i/2}}{g}  \right)\]
as $g\to\infty$.
\end{prp}

\begin{proof} Define $K:=\sum_{i=1}^k r_i$ and
\[A:= \prod_{i=1}^k [a_i,b_i]^{r_i} \subset \mathbb{R}_+^K \]
Using Theorem \ref{thm_integration} we obtain
\[\EE[Y_{g,r_1,\ldots,r_k}] =\frac{1}{V_g} \sum_{[\Gamma]}  C_\Gamma \int_{\mathbb{R}_+^K}\chi_A(x)V_g(\Gamma, x)x_1\cdots x_K dx_1\wedge \cdots \wedge dx_K \]
This sum runs of all $\MCG(\Sigma_g)$-orbits $[\Gamma]$ of (ordered) lists of ordered lists $\Gamma=(\Gamma_1,\ldots,\Gamma_k)$, where $\Gamma_i=(\gamma_{i,1},\ldots,\gamma_{i,r_i})$ is an ordered list of disjoint simple closed curves and $\Gamma_i\cap\Gamma_j=\emptyset$ when $i\neq j$.

We start by singling out one special term in this sum, namely the term that corresponds to non-separating $\Gamma$. If $g$ is large enough so that $\Sigma_g$ allows for $\sum_i r_i$ non-separating simple closed geodesics, then there is exactly one $\MCG(\Sigma_g)$ orbit of such lists. Let us call this orbit $[\Gamma_0]$. Note that $\Sigma_g\setminus\Gamma_0=\Sigma_{g-K,2K}$ and $M(\Gamma_0)=0$. We claim that asymptotically $\EE[Y_{g,r_1,\ldots,r_k}]$ is dominated by the term corresponding to $[\Gamma_0]$. Indeed, by Proposition \ref{prp_volestimate1} we have
\begin{align*}
\int_{A} V_{g-K,2K}(x)x\cdot dx & = V_{g-K,2K}\left(\int_A \bigwedge_{i=1}^K\frac{4\sinh(x_i/2)^2}{x_i} dx_i + O\left(\frac{\prod\limits_{i=1}^k b_i^{2r_i}}{g}\right)\right) \\
& = V_{g-K,2K}\left(\prod_{i=1}^k \lambda_{[a_i,b_i]}^{r_i} + O\left(\frac{\prod\limits_{i=1}^k b_i^{2r_i}}{g}\right)\right)
\end{align*}
Repeated application of Theorem \ref{thm_volestimate2} tells us that
\[\frac{V_{g-K,2K}}{V_g}  = 1 + O\left(\frac{1}{g}\right)  \]
as $g\to\infty$. 

Hence all that remains is to show that the terms corresponding to the other orbits are negligible as $g\to\infty$. We will order these orbits $[\Gamma]$ by how many connected components $\Sigma\setminus\Gamma$ has. Let us call this number $q(\Gamma)$. The integral we need to estimate is:
\[I_{\mathrm{sep}}:=\sum_{q=2}^{K+1}\sum_{\substack{[\Gamma]\;\text{s.t.}\\ q(\Gamma)=q}} C_\Gamma \int_{A}V_g(\Gamma,x)x_1\cdots x_K dx_1\wedge \cdots \wedge dx_K \]
Suppose that $\Sigma_g\setminus\Gamma=\bigsqcup\limits_{i=1}^q\Sigma_{g_i,n_i}$. An Euler characteristic computation tells us that
\[2g-2+2q -2K = \sum_{i=1}^q 2g_i\]
Let us write $A_i = \prod\limits_{k=1}^{n_i}[a_{j_k},b_{j_k}]$. Then Lemma \ref{lem_volestimate3} tells us that
\begin{align*}
\int\limits_{A_i} V_{g_i,n_i}(x_{j_1},\ldots, x_{j_{n_i}})\bigwedge_{k=1}^{n_i} x_{j_k} dx_{j_k} & \leq V_{g_i,n_i}\int\limits_{A_i}  \exp\left(\frac{1}{2}\sum_{k=1}^{n_i} x_{j_k}\right)\bigwedge_{k=1}^{n_i} x_{j_k}dx_{j_k}  \\
& \leq V_{g_i,n_i} \prod_{k=1}^{n_i} b_i^2\exp(b_i/2)
\end{align*}
So we obtain
\[I_{\mathrm{sep}} \leq C_k\prod_{i=1}^K b_i^2 e^{b_i/2} \sum_{q=2}^{K+1} \sum_{\{(g_i,n_i)\}} \frac{(2K)!!}{n_1!\cdots n_q!} V_{g_1,n_1} \times\ldots\times V_{g_q,n_q}  \]
where the inner sum runs over all sequences $\{(g_i,n_i)\}_{i=1}^q$ such that
\[\sum_{i=1}^q g_i = g+q-K-1 \;\text{and}\;\sum_{i=1}^qn_i = 2K\]
and the factor  $\frac{K!!}{n_1!\cdots n_q!}$ accounts for the number of ways the surfaces $\Sigma_{g_i,n_i}$ can be glued into a surface $\Sigma_g$ and $C_k$ is a constant depending on $k$ only. Finally, recall that $K!!=(K-1)(K-3)\cdots 3\cdot 1$. We now apply Lemma \ref{lem_volestimate4} to finish the proof.
\end{proof}

For $Y'$, we will use the following results by Basmajian:
\begin{thm}\label{thm_basmajian}
\cite[Theorem 1.2]{Bas} Let $\gamma$ be a closed geodesic on a hyperbolic surface $X$. Suppose that $\gamma$ has $m\geq 1$ self-intersections. Then the length $\ell(\gamma)$ of $\gamma$ satisfies
\[ \ell(\gamma) \geq \frac{1}{2}\log\left(\frac{m}{2}\right). \]
\end{thm}

\begin{lem}\label{lem_basmajian} \cite[Lemma 2.2]{Bas}  Let $I$ and $J$ be geodesic segments on a hyperbolic surface that intersect in their endpoints. Assume that the intersection at at least one of the endpoints is transversal. Moreover, assume $I$ is contained in a closed geodesics of length $L$. Then the lengths $\ell(I)$ and $\ell(J)$ of these segments satisfy
\[ \ell(I) + \ell(J) \geq 2 \log (\coth(L/4)).\] 
\end{lem}

In \cite{Bas}, the lemma above, which is an application of the Margulis lemma, is in fact stated only for non-simple geodesics. But Basmajian's proof applies to the case of simple geodesics as well.

Using this, we obtain:
\begin{prp}\label{prp_poissonbound2} We have
\[\EE_g[Y'_{g,r_1,\ldots,r_k}] =  O\left( \frac{1}{g}  \right)\]
as $g\to\infty$. Here, the implied constant depends both on $b_1,\ldots,b_k$ and on $r_1,\ldots, r_k$.
\end{prp}

\begin{proof} The idea of the proof is that if a surface contains a set of curves $\Gamma=\{\gamma_1,\ldots,\gamma_k\}$ with self-intersections (either between multiple curves or in a single curve) of total length $L$, then it contains a multicurve of simple curves of length $\leq 2L$ (the boundary of a regular neighborhood of the given curves) that separates off a subsurface (the regular neighborhood). We need to prove that this multicurve contains at least one non-trivial component and, in order to apply the arguments from Proposition \ref{prp_poissonbound1}, we need to have a bound on the complexity of the surface.

First, let us prove that the Euler characteristic of this subsurface is bounded in terms of $b_1,\ldots,b_k$ and $r_1,\ldots,r_k$. To do this, we will first use Theorem \ref{thm_basmajian} and Lemma \ref{lem_basmajian} to show that the total number of intersections in our set of curves is bounded. 

It follows from Theorem \ref{thm_basmajian} that the total number of self-intersections is bounded in terms of $b_1,\ldots,b_k$. So, all that we need is a bound on the number intersections between distinct curves. This can be done in a very similar way to the proof of Theorem \ref{thm_basmajian}, using Lemma \ref{lem_basmajian}. Suppose $\gamma_1$ and $\gamma_2$ are closed geodesics of length $L_1$ and $L_2$ respectively and suppose that $L_1\leq L_2$. Set 
\[A(L) = \log(\coth(L/4)).\]
Then, by dividing $\gamma_1$ into segments $J_1, \ldots J_p$ with disjoint interiors so that all the intersections between $\gamma_1$ and $\gamma_2$ lie in the interiors of the segments $J_i$, $\ell(J_i) = A(L_2)$ for $i=1,\ldots p-1$ and $\ell(J_p) \leq A(L_2)$, we obtain
\[ \card{\gamma_1\cap \gamma_2} = \sum_{i=1}^p \card{J_i\cap \gamma_2}.\]
In order to control $\card{J_i\cap \gamma_2}$, we divide $\gamma_2$ up into segments $I_1,\ldots, I_{m_i}$, so that the endpoints of $I_1$ are the first and second intersection of $\gamma_2$ with $J_i$ and the endpoints of $I_2$ are the second and third and so on. So, in particular $m_i=\card{J_i\cap \gamma_2}$. If $m_i\geq 2$, then because $\ell(J_i)\leq A(L_2)$ we get that 
\[\ell(I_j) \geq A(L_2)\]
from Lemma \ref{lem_basmajian} and hence 
\[\card{J_i\cap \gamma_2} = m_i \leq \max\{L_2 / A(L_2),1\} \leq L_2 / A(L_2)+1.\] This means that
\[ \card{\gamma_1\cap \gamma_2} \leq \left(\frac{L_1}{A(L_2)} + 1\right) \cdot  \left(\frac{L_2}{A(L_2)}+1\right).\]
Putting everything together, the total number of self intersections in our set of curves is at most
\[\sum_{i=1}^k 2\;r_i\;e^{2 b_i} + \binom{\sum_{i=1}^k r_i}{2}  \left(\frac{b_k}{A(b_k)} + 1\right) \cdot  \left(\frac{b_k}{A(b_k)}+1\right) =: C \]

A simple Euler characteristic computation shows that the signature $(g',n')$ can hence be bounded by
\[ 2g'+n'-2 \leq C-1.\]
In particular, for $g$ large enough, the set of curves $\Gamma$ will not be filling and this means that the claim we made in the beginning of the proof (the fact that our set of curves $\Gamma$ gives rise to a separating multicurve) is indeed true.

Moreover, a closed hyperbolic surface of genus $g$ contains at most $(g-1)\cdot e^{L+6}$ closed geodesics with length $\leq L$ that are not iterates of closed geodesics of length $\leq 2 \arcsin 1$ (see \cite[Lemma 6.6.4]{Bus}). So, by doubling, we get that a surface with boundary of signature $(g',n')$ can contain at most $(4g'+2n'-4)\cdot e^{L+6}$ closed geodesics of length $\leq L$ that are not iterates of closed geodesics of length $\leq 2 \arcsin 1$. Moreover, it follows from the collar lemma (see \cite[Theorem 4.1.6]{Bus}) that the total number of geodesics of length $\leq 2 \arcsin 1$ can be bounded in terms of $(g',n')$. Combining these two bounds, we obtain a bound $D=D((b_1,\ldots,b_k),(r_1,\ldots,r_k))$ on the total number of geodesics of configurations of curves that a single separating multicurve can account for.

All in all, this means that the proof reduces to the same argument we used to bound $I_{\mathrm{sep}}$ in the previous proposition. The only difference is that we lose the control over the dependence on length. The reason for this is that the dependence $I_{\mathrm{sep}}$ on the number of components of the separating multicurve  is not explicit, and through our constant $D$, the length influences this number of components.
\end{proof}

\begin{proof}[Proof of Theorem \ref{thm_lengthspectrum}] The theorem follows from the combination of Propositions \ref{prp_poissonbound1}, \ref{prp_poissonbound2} and Theorem \ref{thm_poisson}.
\end{proof}

\section{The systole} \label{sec_systole}

Let us write $\sys:\mathcal{M}_g\to\mathbb{R}_+$ for the function that assigns the length of the systole to a surface. In \cite{Mir3} it was shown that there exists a genus independent constant $C>0$ so that for all $g\in\mathbb{N}$
\[\frac{1}{C}\cdot \varepsilon^2 \leq \PP_g[\sys(X)\leq \varepsilon] \leq C\cdot \varepsilon^2, \]
for all $\varepsilon$ small enough. Let us compare this with Theorem \ref{thm_lengthspectrum}. This tells us that:
\[\lim_{g\to\infty} \PP_g[\sys(X) \leq \varepsilon]  = 1 - e^{-\lambda_{[0,\varepsilon]}} \sim 1-e^{-\frac{\varepsilon^2}{2}}\sim \frac{\varepsilon^2}{2} \]
as $\varepsilon\to 0$.

For the expected systole we obtain the following:
\begin{thm}\label{thm_systole} We have:
\[\lim\limits_{g\to\infty} \EE_g[\sys] = \int_0^\infty e^{-\lambda_{[0,R]}}dR =1.61498\ldots\]
\end{thm}

\begin{proof} For $x\in\mathbb{R}_+$ we have
\[\PP_g[\sys \leq x] = 1 - \PP_g[N_{g,[0,x]}=0] \]
Note that $\PP_g[\sys \leq x] = 1 $ for $x\geq 2\log(4g-2)$ (see for instance \cite[Lemma 5.2.1]{Bus}).

We will first prove that $\PP_g[\sys \leq x]$ is absolutely continuous as a function of $x\in\mathbb{R}_+$. In view of the above, we only need to worry about the interval $[0,2\log(4g-2)]$. So, given $0\leq x_1<y_1\leq x_2<y_2\leq \ldots \leq x_k<y_k \leq 2\log(4g-2)$ such that $\sum\limits_{i=1}^k y_i-x_i < \delta$, we have
\begin{align*}
\sum_{i=1}^k \abs{\PP_g[\sys \leq y_k]- \PP_g[\sys \leq x_k}] & = \sum_{i=1}^k \PP_g[\sys \in [x_k,y_k]]\\
& \leq \sum_{i=1}^k\frac{1}{V_g}\int_{\mathcal{M}_g}N^{\circ}_{g,[x_i,y_i]}(X)dX
\end{align*}
where $N^{\circ}_{g,[x_i,y_i]}$ counts only simple curves with length in $[x_i,y_i]$ (so we use that the systole on a closed surface is always simple). If we use Theorem \ref{thm_integration}, we obtain
\[\sum_{i=1}^k\frac{1}{V_g}\int_{\mathcal{M}_g}N^{\circ}_{g,[x_i,y_i]}(X)dX = \sum_{i=1}^k\sum_{[\gamma]} \frac{1}{V_g}\int_{x_i}^{y_i}V_g(\gamma,t)\;t\;dt \]
where the inner sum runs over (the finite number of) mapping class group orbits $[\gamma]$ of simple curves. Theorem \ref{thm_volformula} now tells us that every term in the sum on the left hand side is given by the same polynomial of degree $6g-2$ in the $x_i,y_i$. Because polynomials are Lipschitz on intervals and all sums are finite, we obtain
\[\sum_{i=1}^k \abs{\PP_g[\sys \leq y_k]- \PP_g[\sys \leq x_k}] < C_g\delta \]
for some constant $C_g>0$ depending only on $g$, which implies that $\PP_g[\sys \leq x]$ is indeed absolutely continuous.

The upshot of this is that we can associate a density function $p_g:\mathbb{R}_+\to \mathbb{R}_+$ to the systole, given by
\[p_g(x)=-\frac{d}{dx}\PP_g[N_{g,[0,x]}=0]  \]
for a.e. $x\in\mathbb{R}^+$. Hence
\begin{align*}
\EE_g[\sys] &  = -\int_0^\infty x \frac{d}{dx}\PP_g[N_{g,[0,x]}=0] dx \\
&  = \left[-x\PP_g[N_{g,[0,x]}=0] \right]_0^\infty + \int_0^\infty \PP_g[N_{g,[0,x]}=0]dx \\
& = \int_0^\infty \PP_g[N_{g,[0,x]}=0]dx.
\end{align*}
Here we have used the fact that, when $g$ is fixed, the systole is uniformly bounded from above to show that
\[\lim_{x\to\infty} x\PP_g[N_{g,[0,x]}=0] = 0\]
In order to prove our statement, we need to show that
\[\lim_{g\to\infty}\int_0^\infty \PP_g[N_{g,[0,x]}=0]dx = \int_0^\infty \lim_{g\to\infty}\PP_g[N_{g,[0,x]}=0]dx\]
To do this, we will apply the dominated convergence theorem. So we need to find a uniform integrable upper bound on $\PP_g[N_{g,[0,x]}=0]$. 

First of all, $\PP_g[N_{g,[0,x]}=0]=0$ for $x\geq 2\log(4g-2)$. So, it follows from Theorem \ref{thm_sysbd} that 
\[ \PP_g[N_{g,[0,x]}=0] \leq \left\{ \begin{array}{ll} B\cdot x \cdot e^{-x} & \text{if}\; x< A\log(g) \\ B\cdot A\log(g) \cdot e^{-A\log(g)} & \text{if}\;  A\log(g) \leq x \leq 2\log(4g-2) \\
0 & \text{otherwise} \end{array}\right. \]
The right hand side can be bounded above by $e^{-A'x}$ for some $A'>0$ and for all $x\in\mathbb{R}_+$. This is an integrable function; as such the dominated convergence theorem applies and this finishes the proof of the theorem.
\end{proof}

\appendix
\section{Random surfaces with large systoles}\label{app_workmaryam}

In this appendix we sketch a proof of the following result due to Mirzakhani. 

\begin{thmrep}{\ref{thm_sysbd}}\cite{Mir4} There exist universal constants $A, B>0$ so that for any sequence $\{c_g\}_g$ of postive numbers with $c_g<A\; \log(g)$ we have:
\[\PP_g[\text{The systole of }S\text{ has length }>c_g] < B\;c_g \; e^{-c_g}.\]
\end{thmrep}

\begin{proof}[Proof sketch] Our phrasing is more in line with that in this paper, but the proof uses the same ideas as that of Mirzakhani. We will denote by $N^*_{g,[x_i,y_i]}:\Mod_g\to\NN$ the number of closed geodesic with length in $[x_i,y_i]$ that are both simple and intersect at most once with any other geodesic with length in $[x_i,y_i]$. Because the systole of a closed hyperbolic surface is always simple and a pair of systoles intersects at most once, we have
\[ \PP_g[\text{The systole of }S\text{ has length }\geq c_g] = \PP_g[N^*_{g,[0,c_g]}=0].\]
The second moment method tells us that
\[ \PP_g[N^*_{g,[0,c_g]}=0] \leq \frac{\EE_g\left[\left(N^*_{g,[0,c_g]}\right)^2\right]-\EE_g\left[N^*_{g,[0,c_g]}\right]^2}{\EE_g\left[\left(N^*_{g,[0,c_g]}\right)^2\right]}.\]
Using that $\left(N^*_{g,[0,c_g]}\right)^2 = \left(N^*_{g,[0,c_g]}\right)_2 + N^*_{g,[0,c_g]}$, we obtain that
\[ \PP_g[N^*_{g,[0,c_g]}=0] \leq \frac{\EE_g\left[\left(N^*_{g,[0,c_g]}\right)_2\right]+\EE_g\left[N^*_{g,[0,c_g]}\right]-\EE_g\left[N^*_{g,[0,c_g]}\right]^2}{\EE_g\left[\left(N^*_{g,[0,c_g]}\right)_2\right] + \EE_g\left[N^*_{g,[0,c_g]}\right]}.\]
To get a bound on $\EE_g\left[\left(N^{*}_{g,[0,c_g]}\right)_2\right]$, we can argue like in Propositions \ref{prp_poissonbound1} and \ref{prp_poissonbound2}. However, since we are now only considering curves that pairwise intersect at most once. The surface that is filled by two such intersecting curves is a one holed torus. So, by replacing the constant $C$ in Proposition \ref{prp_poissonbound2} by $1$, we obtain that
\[\EE_g\left[\left(N^{*}_{g,[0,c_g]}\right)_2\right] = \lambda_{[0,c_g]}^2 + O\left(\frac{\prod\limits_{i=1}^k c_g^4 e^{2 c_g}}{g}  \right) \]
and
\[\EE_g\left[N^{*}_{g,[0,c_g]}\right] = \lambda_{[0,c_g]}+O\left(\frac{\prod\limits_{i=1}^k c_g^2 e^{c_g}}{g}  \right).\]
filling this in and using our assumption on $c_g$ gives that
\[ \PP_g[\text{The systole of }S\text{ has length }\geq c_g] \leq C_1\; \frac{1}{\lambda_{[0,c_g]}+C_2} \]
for some constant $C_1,C_2>0$. The fact that 
\[\lambda_{[0,x]} \sim \frac{e^{x}}{x}\]
as $x\to \infty$, now implies the result.
\end{proof}

%%%%%%%%%%%%%%%%%%%%%%%%%%%%%%%%%%%%%%%%%%%%%%%%%
%		R E F E R E N C E S
%%%%%%%%%%%%%%%%%%%%%%%%%%%%%%%%%%%%%%%%%%%%%%%%%

%\nocite{*}

\bibliographystyle{alpha}
\bibliography{RandSurf.bib}

\vspace{2cm}
\begin{flushright}
\small{Institut de Math\'ematiques de Jussieu-Paris Rive Gauche,\\
Sorbonne Universit\'e Paris, France \\
E-mail: bram.petri@imj-prg.f}
\end{flushright}

\end{document}